\documentclass{article}
\usepackage{graphicx} 
\usepackage{amsfonts,amsmath, amssymb, amsthm,xcolor}

\title{Helly-type theorems for monotone properties of boxes}
\author{N\'ora Frankl\footnote{School of Mathematics and Statistics, The Open University, UK, nora.frankl@open.ac.uk} \hspace{0.1em} and Attila Jung\footnote{ELTE E\"{o}tv\"{o}s Lor\'{a}nd University and HUN-REN Alfr\'ed R\'enyi Institute of Mathematics, Hungary, jungattila@gmail.com}}
\date{May 2024}

\newtheorem{theorem}{Theorem}[section]
\newtheorem{lemma}[theorem]{Lemma}
\newtheorem{claim}[theorem]{Claim}
\newtheorem{prop}[theorem]{Proposition}

\newtheorem{cor}[theorem]{Corollary}

\newcommand{\B}{\mathcal{B}}
\newcommand{\F}{\mathcal{F}}

\renewcommand{\Re}{\mathbb{R}}

\begin{document}
	
	\maketitle
	
	\begin{abstract}
		We present a unified approach to prove Helly-type theorems for monotone properties of boxes, such as having large volume or containing points from a given set. As a corollary, we obtain new proofs for several earlier results regarding specific monotone properties. Our results generalise to $H$-convex sets as well.
	\end{abstract}

	\section{Introduction}\label{sec:intro}
	Helly's theorem \cite{helly} states that if in a family of convex sets in $\Re^d$ the intersection of any $d+1$ members is nonempty, then the intersection of the whole family is nonempty. This is a  cornerstone result in Discrete Geometry with many generalisations and extensions, see for example the survey~\cite{barany2022helly}.
	
	For special families of convex sets often stronger results hold. Among such families, axis parallel boxes (called simply \emph{boxes} from now on) received particular attention \cite{katchalski1980boxes, danzer1982intersection, eckhoff, eckhoff1991intersection, soberon}. It is folklore that requiring non-empty pairwise intersections in a family of boxes guarantees a common point in the whole family, regardless of the dimension \cite{eckhoff}. Indeed, project the family to each coordinate axes, and apply Helly's theorem in each of the $d+1$ families of intervals obtained this way.

	The quantitative volume theorem of B\'ar\'any, Katchalski and Pach~\cite{barany1982quantitative} states that if in a finite family $\F$ of convex sets the intersection of any $2d$ members has volume at least $1$, then the intersection of the whole family is of volume at least $v(d)$, for some positive function $v(d)$. Nasz\'odi proved that $v(d)$ can be chosen to be of order $d^{-cd}$ \cite{naszodi2016proof}. Although the number $2d$ is best possible even for families of boxes, the volume bound can be improved to $v(d)=1$ in this case. This is a direct corollary of the following well-known Lemma.
	
	\begin{lemma}\label{lem:strongHellyBoxes}
		For any finite family $\F$ of boxes in $\Re^d$, there is a subfamily $\F' \subset \F$ of size at most $2d$ such that $\cap \F = \cap \F'$.
	\end{lemma}
	
	We call a property $P$ of boxes a \emph{monotone property} if $B \in P$ and $B \subset C$ implies $C \in P$. Lemma \ref{lem:strongHellyBoxes} implies results not only with respect to the volume, but also with respect to any monotone property.
	
	\begin{cor}\label{cor:HellyBoxesP}
		For any monotone property $P$ and any finite family $\F$ of boxes in $\Re^d$ if the intersection of every subfamily of $\F$ of size $2d$ has property $P$, then the intersection of all the members of $\F$ has property $P$.
	\end{cor}

	If $P$ is the property of containing a point from a given finite set, we obtain discrete Helly theorems for boxes proved by Halman \cite{halman} and in more general forms by Edwards and Sober\'on \cite{soberon}.  Proofs of various earlier Helly-type results for boxes were specific to given monotone properties. In this note we present results similar to Lemma \ref{lem:strongHellyBoxes} from which many of these results follow at once.
	
	We call a family $\B$ of boxes \emph{$P$-intersecting} if the intersection of its members has property $P$. In this language Corollary~\ref{cor:HellyBoxesP} states that if every subfamily of size $2d$ is $P$-intersecting, then the whole family is $P$-intersecting.
	
	\subsubsection*{Colourful results}\label{subsec:color}
	
	Let $\F_1, \ldots, \F_{d+1}$ be $d+1$ finite families of convex sets. The Colourful Helly theorem of Lov\'asz (first published by B\'ar\'any \cite{barany1982generalization}) states that if for any choice $C_i \in \F_i$ we have $\cap_{i=1}^{d+1}C_i \neq \emptyset$, then there exists an $i \in [d+1]$ such that $\cap \F_i \neq \emptyset$. As the families are not necessarily distinct, this result generalises Helly's theorem. The Colourful Helly theorem is optimal in the sense that an analogous statement is not true for fewer than $d+1$ families. This can be shown by a construction using hyperplanes in general position.
	
	We show the following strong intersection property of boxes, which generalises Lemma \ref{lem:strongHellyBoxes} and implies a Colourful Helly theorem for monotone properties.

	\begin{theorem}\label{thm:colourfulStrongHelly}
		Let $\mathcal{B}_1,\dots,\mathcal{B}_{2d}$ be finite families of boxes in $\mathbb{R}^d$. Then there is a  selection $B_i\in \mathcal{B}_i$ for each $i\in [2d]$ and an index $\ell \in [2d]$ such that $\cap_{i=1}^{2d}B_i\subset \cap \mathcal{B}_{\ell}$. 
	\end{theorem}
	
	\begin{cor}\label{cor:colourfulHellyP}
		Let $P$ be any monotone property of boxes and let $\mathcal{B}_1,\dots,\mathcal{B}_{2d}$ be finite families of boxes in $\mathbb{R}^d$. If for every choice $B_i\in \mathcal{B}_i$ the family $\{B_1, \ldots, B_{2d}\}$ is $P$-intersecting, then there exists an $\ell \in [2d]$ such that $\B_\ell$ is $P$-intersecting.
	\end{cor}

	The following construction shows that the parameter $2d$ is best possible in Corollary \ref{cor:colourfulHellyP} when $P$ is the property of having volume at least $1$ and all the families are the same, hence best possible in Theorem \ref{thm:colourfulStrongHelly}. Let $\F$ be a family of $2d$ halfspaces whose intersection is a box of volume $\varepsilon < 1$. Intersecting these halfspaces with a large cube, we obtain a family of $2d$ boxes such that the intersection of any $2d-1$ of them has volume at least $1$, but the intersection of all of them has volume $\varepsilon < 1$. 
	Dam\'asdi, F\"oldvári and Naszódi~\cite{damasdi2021colorful} proved a volumetric version of the Colourful Helly theorem for $3d$ families of convex sets. If property $P$ is having volume at least $1$, Corollary \ref{cor:colourfulHellyP} implies a volumetric Colourful Helly theorem for $2d$ families of boxes with volume bound $v(d)=1$. If $P$ is containing at least $n$ points from a given finite set, we obtain a result of Edwards and Sober\'on \cite[Theorem 1.2]{soberon}.

	\subsubsection*{Fractional results}\label{subsec:frac}
	
	The fractional Helly theorem of Katchalski and Liu~\cite{katchalski1979problem} generalises Helly's theorem by showing that if $\F$ is a finite family of convex sets from $\Re^d$ such that for some $\alpha > 0$, at least $\alpha \binom{|\F|}{d+1}$ of the $(d+1)$-tuples of $\F$ have a nonempty intersection, then there is a subfamily $\F' \subset F$ of size $|\F'| \geq \beta(\alpha, d)|\F|$ with nonempty intersection. Variants for general convex sets include a version with respect to containing lattice points by Bárány and Matou\v sek \cite{barany2003fractional},  or having large volume first proved by Sarkar, Xue and Sober\'on \cite{sarkar2021quantitative} for intersecting $\frac{d(d+3)}{2}$-tuples and improved in \cite{frankl2024quantitative} for intersecting $(d+1)$-tuples.
	
	We prove two variants for boxes with respect to monotone properties. The first of these is stronger in the sense that it assumes that a positive fraction of the $d+1$ tuples has property $P$, as opposed to assuming it for $2d$ tuples in the second. However, in the second version $\beta$ is more optimal.
	
	\begin{theorem}\label{thm:FHboxesd+1}
		For every dimension $d$ and every real number $\alpha > 0$ there exists $\beta > 0$ such that the following holds.
		Let $P$ be a monotone property of boxes, and let $\mathcal{B}$ be a finite family of boxes in $\mathbb{R}^d$. If at least $\alpha \binom{|\B|}{d+1}$ of the $(d+1)$-tuples in $\mathcal{B}$ are $P$-intersecting, then there exists a $P$-intersecting subfamily $\mathcal{B}'\subseteq \mathcal{B}$ of size at least $\beta |\B|$.
	\end{theorem}
	
	If $P$ is the property of containing a point from a fixed set, Theorem \ref{thm:FHboxesd+1} implies Theorem 1.3 of \cite{soberon} with a shorter proof.
	For the same property, Theorem 3.2 of \cite{soberon} is a fractional result with parameter $2d$. Although the bound for $\beta$ in our Theorem \ref{thm:FHboxes2d} is slightly weaker, it still converges to $1$ if $\alpha$ does so. 
	
	\begin{theorem}\label{thm:FHboxes2d}
		For every dimension $d$ and every real number $\alpha > 0$ there exists $\beta > 0$ such that the following holds.
		Let $P$ be a monotone property of boxes, and let $\mathcal{B}$ be a finite family of boxes in $\mathbb{R}^d$ such that at least $\alpha\binom{|\mathcal{B}|}{2d}$ of the $2d$-tuples of $\mathcal{B}$ are $P$-intersecting. Then there is a $P$-intersecting subfamily $\mathcal{B'}\subseteq \mathcal{B}$ of cardinality at least $\beta(\alpha, k) |\mathcal{B}|$. For large enough $|\B|$, we can choose $\beta=1-2d(1-\alpha)^{1/(2d+1)}$.
	\end{theorem}
	
	Unlike in the case of Helly's theorem, the parameter $(d+1)$ in the fractional version is optimal for boxes for a general $\alpha$. Indeed, let $\mathcal{F}_i$ be a family of $\lfloor n/d \rfloor $ or $\lceil n/d \rceil$ hyperplanes orthogonal to the $i$-th coordinate axis. Intersecting the union of these families with a large cube, we obtain a family of $n$ boxes where $d+1$ cannot be lowered for $\alpha=\binom{n}{d}/(n/d)^d$. By using thick hyperplanes, we obtain a construction with proper $d$-dimensional boxes. However, Katchalski \cite{katchalski1980boxes} showed that when $\alpha$ is sufficiently large, we can decrease $d+1$ to $2$. Going beyond this, Eckhoff \cite{eckhoff}, confirming a conjecture of Kalai \cite{kalai1984intersection}, determined for each $1\leq k \leq d+1$ the smallest $\alpha$ for which if $\alpha \binom{|\F|}{k}$ many $k$-tuples of members of a family of boxes $\mathcal{F}$ intersect, then there is a subset $\F'\subseteq \F$ with $|\F'|\geq \beta(\alpha,d) |\F|$ such that $\cap \F'\neq \emptyset$. As a first step towards generalising Eckhoff's result to monotone properties, we prove an analogue of Katchalki's result.

	\begin{theorem}\label{thm:WFH2}
		For every dimension $d$ there exists a $c_d > 0$ such that for every $\alpha \in (1-c_d, 1]$ there exists a $\beta > 0$ such that the following holds. Let $P$ be a monotone property of boxes, and let $\mathcal{B}$ be a finite family of boxes in $\mathbb{R}^d$. If at least $\alpha \binom{|\B|}{2}$ of the pairs in $\mathcal{B}$ are $P$-intersecting, then there exists a $P$-intersecting subfamily $\mathcal{B}'\subseteq \mathcal{B}$ of size at least $\beta |\B|$.
	\end{theorem}
	
	We note that in Katchalski's result the values $c_d$ and $\beta$ are best possible, whereas we don't have any reason to think that they are optimal in our result.
	
	During the preparation of our manuscript, a new work of Eom, Kim and Lee appeared where they prove Theorem~\ref{thm:WFH2} independently in the case when $P$ is the property that a box contains an element of a fixed finite sets $S \subset \Re^d$ \cite{eom2025fractionaldiscretehellypairs}.
	
	\subsubsection*{A $(p,q)$-theorem}
	The celebrated Alon-Kleitman $(p,q)$-theorem \cite{alon1992piercing} says that for every $p \geq q \geq d+1$ there exists $c = c(p,q,d)$ such that the following holds. If $\F$ is a finite family of convex sets in $\Re^d$ such that among any $p$ sets, some $q$ have a point in common, then there is a set $S \subset \Re^d$ of size at most $c$ such that for every $C \in \F$ we have $S \cap C \neq \emptyset$. Theorem~\ref{thm:FHboxesd+1} implies an analogue of the celebrated $(p,q)$-theorem for monotone properties of boxes.
	
	\begin{theorem}\label{thm:pqboxes}
		For every $p \geq q \geq d+1$ there exists $c = c(p,q,d)$ such that the following holds. Let $P$ be a monotone property and $\B$ be a finite family of boxes in $\Re^d$. If among every $p$ box, some $q$ are $P$-intersecting, then there exists a family $\mathcal{S}$ of boxes with property $P$ such that $|\mathcal{S}| \leq c$ and for every $B \in \B$ there is a $D \in \mathcal{S}$ with $D \subset B$.
	\end{theorem}

	\subsubsection*{$H$-convex sets}
	
	Let $H=\{H_1,\dots,H_k\}$ be a set of $k$ closed halfspaces in $\mathbb{R}^d$. An \emph{$H$-convex set} $B$ in $\mathbb{R}^d$ is the intersection of $k$ halpfspaces $H_1^B,\dots,H_{k}^B$ such that for each $i\in [k]$ the halfspace $H_i^B$ is a translate of $H_i$. The notion of $H$-convex sets was introduced by Boltyanski \cite{Bolt}, and their Helly-type properties have been studied in \cite{boltyanskii1976helly, boltyanski2003minkowski, soberon}.\footnote{In the cited papers the notion of $H$-convex sets also requires that the normal vectors of the hyperplanes bounding $H_1,\dots,H_k$ are not contained in a closed hemisphere. However, our proofs work without this assumption too.}

	Boxes are special families of $H$-convex sets, where each halfspace is a coordinate-halfspace. Our results for boxes can be stated and proved essentially in the same way for $H$-convex sets as well, and they directly imply the results about boxes We will state and prove these more general results in the next section. Theorem\ref{thm:colourfulStrongHelly} follows from Theorem \ref{thm:HcolourfulStrongHelly}, Theorem \ref{thm:FHboxesd+1} from Theorem \ref{thm:HFHboxesd+1}, Theorem \ref{thm:FHboxes2d} from Theorem \ref{thm:HFHboxes2d}, Theorem \ref{thm:WFH2} from Theorem \ref{thm:HWFH2}, and Theorem \ref{thm:pqboxes} from Theorem \ref{thm:Hpqboxes}.

	\vspace{20pt}
	
	\textbf{Acknowledgement:} NF was partially supported by ERC Advanced Grant no. 882971 “GeoScape” at the beginning of this work. AJ was supported by the ERC Advanced Grant ``ERMiD'', by the EXCELLENCE-24 project no.~151504 of the NRDI Fund, and by the NKFIH grants TKP2021-NKTA-62, FK132060 and SNN135643.

	\section{Results for $H$-convex sets and their proofs}\label{sec:Hconvex}
	
	The generalisations of our results to $H$-convex sets illustrate well that they are more combinatorial than geometric, as they depend only on the number of defining halfspaces, and not on the dimension of the space. The generalisation of Theorem \ref{thm:colourfulStrongHelly} to $H$-convex sets is as follows.
	
	\begin{theorem}\label{thm:HcolourfulStrongHelly}
		Let $H$ be a set of $k$ closed halfspaces, and let $\mathcal{B}_1,\dots,\mathcal{B}_{k}$ be finite families of $H$-convex sets in $\mathbb{R}^d$. Then there is a  selection $B_i\in \mathcal{B}_i$ for each $i\in [k]$ and an index $\ell \in [k]$ such that $\cap_{i=1}^{k}B_i\subset \cap \mathcal{B}_{\ell}$. 
	\end{theorem}
	
	Our proofs rely on introducing multiple orderings on $H$-convex sets. Let $H=\{H_1,\dots,H_k\}$ be a set of $k$ closed halfspaces in $\mathbb{R}^d$. We define $k$ orderings $<_1,\dots, <_{k}$ on the set of $H$-convex sets as follows. For two $H$-convex sets $B_1,B_2$ we set $B_1 <_i B_2$ if $H_i^{B_1}\subseteq H_i^{B_2}$. Using these orderings, the proof of Theorem \ref{thm:HcolourfulStrongHelly} is very short.

	\begin{proof}[Proof of Theorem~\ref{thm:HcolourfulStrongHelly}]
		Sequentially define a permutation $\pi\in S_{k}$ such that $\mathcal{B}_{\pi(i)}$ contains a minimal element $B_{\pi(i)}$ of $\cup_{j=1}^n\mathcal{B}_j\setminus (\cup_{j=1}^{i-1}\mathcal{B}_{\pi(j)})$ according to $<_i$. Then $\cap_{i=1}^{k}B_{\pi(i)}\subseteq \cap\mathcal{B}_{\pi(k)}$, so we can choose $\{B_{\pi(1)},\ldots,B_{\pi(k)}\}$ and $\ell=\pi(k)$.
	\end{proof}
	
	Similarly to boxes, we call a property $P$ of $H$-convex sets a \emph{monotone property} if $B \in P$ and $B \subset C$ implies $C \in P$. A family $\B$ of $H$-convex sets is \emph{$P$-intersecting} if the intersection of its members has property $P$. Of Theorem \ref{thm:FHboxesd+1} and Theorem \ref{thm:FHboxes2d} we generalise and prove Theorem \ref{thm:FHboxes2d} first, as its proof is shorter.
	
	\begin{theorem}\label{thm:HFHboxes2d}
		For every integer $k$ and every real number $\alpha > 0$ there exists $\beta > 0$ such that the following holds.
		Let $H$ be a set of $k$ closed halfspaces in $\mathbb{R}^d$, let $P$ be a monotone property of $H$-convex sets, and let $\mathcal{B}$ be a finite family of $H$-convex sets in $\mathbb{R}^d$ such that at least $\alpha\binom{|\mathcal{B}|}{k}$ of the $k$-tuples of $\mathcal{B}$ are $P$-intersecting. Then there is a $P$-intersecting subfamily $\mathcal{B'}\subseteq \mathcal{B}$ of cardinality at least $\beta|\mathcal{B}|$. For large enough $|\B|$, we can choose $\beta=1-k(1-\alpha)^{1/(k+1)}$.
	\end{theorem}
	
	\begin{proof}[Proof of Theorem~\ref{thm:HFHboxes2d}] We may assume that $|\B|$ is large enough compared to $\alpha$ and $k$, otherwise we can choose $\beta = \frac{1}{|\B|}$. Let $\gamma=(1-\alpha)^{1/(k+1)}$. For each $i$  let $\mathcal{A}_i\subseteq \mathcal{B}$ be the set of the first $\gamma|\mathcal{B}|$ members of $\B$ according to $<_i$. Then $|\{\{A_1,\dots,A_{2d}\}: A_i\in \mathcal{A}_i\}\}|\geq \frac{1}{k!}|A_1|\cdot (|A_2|-1)\dots (|A_{k}|-k)=\binom{\gamma |\B|}{k} > (1-\alpha) \binom{|\B|}{k}$ as $|\B|$ is large enough. Thus by assumption, there is an $H$-convex set $A$ in $\{\cap_{i=1}^{k} A_i: A_i\in \mathcal{A}_i\}$ that has property $P$. By the properties of the orderings, $A\subseteq B$ for any $B\in \mathcal{B}':=\mathcal{B}\setminus (\cup_{i=1}^{k}\mathcal{A}_i)$, and $|\mathcal{B}'|\geq (1-k\gamma)|\mathcal{B}|=(1-k(1-\alpha)^{1/(k+1)})|\mathcal{B}|$.
	\end{proof}

	The next result is a generalisation of Theorem~\ref{thm:FHboxesd+1}. Note that although we only state it in the case when the number of halfspaces in $H$ is odd ($2k+1$), the even case with $2k$ halfspaces and $k+1$ families of $H$-convex sets follows from directly this.
	
	\begin{theorem}\label{thm:HFHboxesd+1}
		For every integer $k$ and every real number $\alpha > 0$ there exists $\beta > 0$ such that the following holds.
		Let $H$ be a set of $2k+1$ closed halfspaces, let $P$ be a monotone property of $H$-convex sets, and let $\mathcal{B}$ be a finite family of $H$-convex sets in $\mathbb{R}^d$. If at least $\alpha \binom{|\B|}{k+1}$ of the $k+1$-tuples in $\mathcal{B}$ are $P$-intersecting, then there exists a $P$-intersecting subfamily $\mathcal{B}'\subseteq \mathcal{B}$ of size at least $\beta |\B|$.
	\end{theorem}
	
	In the related result of \cite{soberon} (Theorem 5.5) a positive fraction of intersecting $d+1$-tuples is needed, regardless of the number of halfspaces in $H$. Thus, our result is stronger for a different range of parameters.
	
	The framework of the proof of Theorem~\ref{thm:HFHboxesd+1} is based on ideas from \cite{barany2003fractional} and uses a supersaturation result of Erd\H os and Simonovits. The drawback of this widely used method is that the bound it gives on $\beta$ is very weak. 
	Let $K_r(t)$ be the complete $r$-partite $r$-uniform hypergraph with vertex classes of size $t$.
	
	\begin{lemma}[Erd\H os, Simonovits \cite{erdHos1983supersaturated}]\label{lem:supersat}
		Let $r,t$ be positive integers. For every $\alpha> 0$ there exists $\gamma(\alpha, r,t)>0$ such that the following holds. If $G$ is a $r$-uniform hypergraph on $n$ vertices with at least $\alpha\binom{n}{r}$ edges, then $G$ contains at least $\gamma n^{rt}$ copies of $K_r(t)$.
	\end{lemma}
	
	The second ingredient is a variant of Theorem~\ref{thm:HcolourfulStrongHelly}, which is stronger in the sense that the number of families is significantly smaller, however the price of this is a somewhat weaker conclusion. 
	
	\begin{theorem}\label{thm:HWCHB}[Weak Colourful Helly for $H$-convex sets]
		Let $H$ be a set of $2k+1$ closed halfspaces and let $\mathcal{B}_1,\dots,\mathcal{B}_{k+1}$ be finite families of $H$-convex sets in $\mathbb{R}^d$ with $|\mathcal{B}_1|=\dots=|\mathcal{B}_{k+1}|\geq 2^{k(2k+1)+1}$. Then there is a selection $B_i\in \mathcal{B}_i$ for each $i\in[k+1]$ and an index $\ell \in [k+1]$  such that there is a subset $\mathcal{B}'_{\ell}\subseteq \mathcal{B}_{\ell}$ with $|\mathcal{B}'_{\ell}|\geq 2^{-k(2k+1)-1}|\mathcal{B}_{\ell}|$ and $\cap_{i=1}^{k+1}B_i\subset \cap \mathcal{B'}_{\ell}$.
	\end{theorem}

	We say that two ordered pairs of $H$-convex sets $(B_1,B_2)$ and $(B'_1,B'_2)$ are \emph{ordered consistently}, if for each $i\in [k]$, the relation $B_1 <_i B_2$ holds if and only if $B'_1 <_i B'_2$. An ordered set $\B = (B_1, \ldots, B_{\ell})$ of $H$-convex sets is ordered consistently, if for any $i < j$ and $i' < j'$ the ordered pairs $(B_i, B_j)$ and $(B_{i'}, B_{j'})$ are ordered consistently. An ordered pair $(\F_1, \F_2)$ of families of $H$-convex sets is ordered consistently, if for any $B_1, B_1' \in \F_1$, $B_2, B_2' \in \F_2$ the pairs $(B_1, B_2), (B_1', B_2')$ are ordered consistently.

	\begin{claim}\label{cla:Hstructure}Let $H$ be a set of $k$ closed halfspaces and let $\mathcal{B}_1$ and $\mathcal{B}_2$ be two families of $H$-convex sets in $\mathbb{R}^d$ with $|\B_1|=m$, $|\B_2|=n$. Then for any $h\in [k]$ there exist $\mathcal{B}'_1\subseteq \mathcal{B}_1$ and $\mathcal{B}'_2\subseteq\mathcal{B}_2$ with $|\mathcal{B}_1'| = \lfloor\frac{m}{2}\rfloor$, $|\mathcal{B}_2'|=\lfloor\frac{n}{2}\rfloor$ such that we either have $B_1 <_{h} B_2$ for all $B_1\in \mathcal{B}'_1, B_2\in \mathcal{B}'_2$, or we have $B_2 <_{h} B_1$ for all $B_1\in \mathcal{B}'_1, B_2\in \mathcal{B}'_2$.
	\end{claim}
	
	\begin{proof}Let $X\subseteq \mathcal{B}_1\cup \mathcal{B}_2$ be the first $\lfloor\frac{m+n}{2}\rfloor$ members of $\mathcal{B}_1\cup \mathcal{B}_2$ according to $<_{h}$, and let $Y=(\B_1\cup \B_2)\setminus X$. We either have $|X\cap \B_1|\geq \lfloor\frac{m}{2}\rfloor$ and $|Y\cap \B_2|\geq \lfloor\frac{n}{2}\rfloor$, or $|X\cap \B_2|\geq \lfloor\frac{m}{2}\rfloor$ and $|Y\cap \B_1|\geq \lfloor\frac{n}{2}\rfloor$. In the first case $\B'_1\subset X\cap \B_1$ and $\B'_2 \subset Y\cap \B_2$, in the second case $\B'_1\subset Y\cap \B_1$, $\B'_2\subset X\cap \B_2$ of sizes $\lfloor\frac{m}{2}\rfloor$ and $\lfloor\frac{n}{2}\rfloor$ have the desired property.
	\end{proof}

	\begin{proof}[Proof of Theorem \ref{thm:HWCHB}]
		By discarding at most half of the elements from each $\B_i$, we may assume that the size of each of the families is a power of $2$ with an exponent at least $k(2k+1)$. Applying Claim \ref{cla:Hstructure} successively to each pair $\mathcal{B}_i,\mathcal{B}_j$ and each relation $<_{h}$, we obtain subfamilies $\mathcal{B}'_1\subseteq \mathcal{B}_1,\dots,\mathcal{B}'_{k+1}\subseteq \mathcal{B}_{k+1}$ with $|\mathcal{B}'_i|\geq 2^{-k(2k+1)}|\B_i|$ such that each pair $(\B_i',\B_j')$ is ordered consistently.
		By discarding the excess, we may assume that $|\mathcal{B}'_i|= 2^{-k(2k+1)}|\B_i|$.
		
		Let $\B'=\cup_{i=1}^{k+1}\B_i'$. For each $h\in [2k+1]$ there is an $i$ such that $\B'_i$ consists of exactly the first $\frac{1}{k+1}|\B'|$ members of $\B'$ according to $<_h$. Since there are $2k+1$ orderings and $k+1$ families, there is an $\ell\in [k+1]$ such that $\B_{\ell}$ is the family consisting of the first $\frac{1}{k+1}|\B'|$ members of $\B'$ according to at most one ordering, say $<_1$. Let $B_{\ell}\in \B_{\ell}$ be the minimal member of $\B_{\ell}$ according to $<_1$, and for all $i\neq \ell$ let $B_{i}\in \B_{i}$ be an arbitrary element. Since for all $i\in [2k+1]$, $i\neq 1$ there is a $j$ such that all members of $\B_j$ are smaller than all members of $\B_{\ell}$ according to $<_i$, and $B_{\ell}$ is the minimal element of $\B_{\ell}$ according to $<_1$, we have $\cap_{i=1}^{k+1}B_i\subset \cap \mathcal{B}_{\ell}$.
	\end{proof}
	
	\begin{proof}[Proof of Theorem~\ref{thm:HFHboxesd+1}]
		Let $G$ be the $(k+1)$-uniform hypergraph with vertex set $\B$ where hyperedges correspond to $P$-intersecting $(k+1)$-tuples of $H$-convex sets. Let $t = 2^{k(2k+1)+1}(2k+1)$. As $G$ has at least $\alpha\binom{|\B|}{k+1}$ edges, we can apply Lemma~\ref{lem:supersat} to find at least $\gamma  |\B|^{(k+1)t}$ copies of $K_{k+1}(t)$ in $G$. By Theorem~\ref{thm:HWCHB}, for every copy of $K_{k+1}(t)$ we have a $P$-intersecting $(2k+1)$-tuple of $H$-convex sets. Since a $(2k+1)$-tuple can appear in at most $\binom{|\B|-2k-1}{t(k+1)-2k-1}$ different copies of $K_{k+1}(t)$, we have at least
		\[
		\frac{\gamma|\B|^{(k+1)t}}{\binom{|\B|-2k-1}{t(k+1)-2k-1}} \geq \alpha'\binom{|\B|}{2k+1}
		\]
		$P$-intersecting $(2k+1)$-tuples in $\B$ for some $\alpha'=\alpha'(\alpha, k)$. We can apply Theorem~\ref{thm:HFHboxes2d} to find a $P$-intersecting subfamily $\B' \subset \B$ of size $\beta|\B|$.
	\end{proof}

	We generalise Theorem \ref{thm:HWFH2} to $H$-convex sets as follows.
	
	\begin{theorem}\label{thm:HWFH2}
		For every integer $k$ there is a $c_k>0$ such that for every $\alpha \in (1-c_d, 1]$ there exists a $\beta > 0$ such that the following holds. Let $H$ be a set of $k$ closed halfspaces, let $P$ be a monotone property of $H$-convex sets, and let $\mathcal{B}$ be a finite family of $H$-convex sets in $\mathbb{R}^d$. If at least $\alpha \binom{|\B|}{2}$ of the pairs in $\mathcal{B}$ are $P$-intersecting, then there exists a $P$-intersecting subfamily $\mathcal{B}'\subseteq \mathcal{B}$ of size at least $\beta |\B|$.
	\end{theorem}
	
	\begin{proof}[Proof of Theorem~\ref{thm:HWFH2}]
		By applying the Erd\H os-Szekeres theorem successively for each ordering, we obtain an $N=N(k)$ such that in any set of $H$-convex sets of size $N$ there are $k$ which are ordered consistently. Thus there are at least
		\[
		\frac{\binom{|\B|}{N}}{\binom{|\B|-k}{N-2d}} =\frac{(N-k)!(k)!}{N!}\binom{|\B|}{k}\]
		distinct subfamilies of $\mathcal{B}$ of size $k$ which are ordered consistently.

		By setting $c_k=\frac{(N-k)!(k)!}{N!}$, and $\alpha'=\alpha+\frac{(N-k)!(k)!}{N!}-1$, we obtain at least $\alpha'\binom{|\B|}{k}$ many consistently ordered $k$ tuples among which the intersection of any pair has property $P$.
		The following proposition implies that the intersection of any of these $k$-tuples has property $P$.
		
		\begin{prop}
			If $\B=\{B_1 <_1 \ldots <_1 B_{m}\}$ is ordered consistently, then $\cap_{i=1}^{k} B_i=B_1 \cap B_m $.
		\end{prop}
		
		\begin{proof}
			As $\B$ is ordered consistently, for every $i \in [k]$ either $H^{B_1}_i \subset H^{B_2}_i \subset \ldots \subset H^{B_k}_i$ or $H^{B_1}_i \supset H^{B_2}_i \supset \ldots \supset H^{B_k}_i$ holds. Let $j(i) = 1$ in the first case when $H^{B_1}_i \subset H^{B_k}_i$ and let $j(i) = m $ in the second when $H^{B_1}_i \supset H^{B_k}_i$. With this notation
			\[
			\bigcap_{\ell = 1}^k B_\ell = \bigcap_{\ell = 1}^k\bigcap_{i = 1}^{k}H^{B_\ell}_i = \bigcap_{i = 1}^{k}\bigcap_{\ell = 1}^kH^{B_\ell}_i = \bigcap_{i = 1}^{k}H^{B_{j(i)}}_i = B_1 \cap B_m.
			\]
		\end{proof}
		
		Thus, the intersection of $\alpha'\binom{|\B|}{k}$ of the $k$-tuples of $\mathcal{B}$ have property $P$ and we can apply Theorem~\ref{thm:HFHboxes2d}.
	\end{proof}
	
	Our last result is the generalisation of Theorem \ref{thm:pqboxes} to $H$-convex sets. Similarly to the other results, it does not depend on the dimension of the space, only on the number of halfspaces in $H$.
	
	\begin{theorem}\label{thm:Hpqboxes}
		For every $p \geq q \geq k+1$ there exists $c = c(p,q,k)$ such that the following holds. Let $H$ be a set of $k$ closed halfspaces, let $P$ be a monotone property of $H$-convex sets, and let $\B$ be a finite family of $H$-convex sets in in $\Re^d$. If among every $p$ members of $\B$, some $q$ are $P$-intersecting, then there exists a family $\F$ of $H$-convex sets with property $P$ such that $|\F| \leq c$ and for every $B \in \B$ there is a $D \in \F$ with $D \subset B$.
	\end{theorem}

	We only include a brief outline of the proof of Theorem \ref{thm:Hpqboxes}, as the proof follows a method which is very standard by now. It is based on the argument for Euclidean convex sets, which was generalised for abstract set systems by Alon, Kalai, Matou\v sek and Meshulam \cite{alon2002transversal}.
	A combination of \cite[Theorem 8]{alon2002transversal} and \cite[Theorem 9]{alon2002transversal} states that if a set system satisfies the fractional Helly property, then a $(p,q)$-theorem holds for the same set system. In order to prove Theorem~\ref{thm:Hpqboxes} from Theorem~\ref{thm:HFHboxesd+1}, one can use the following set system. Let the base set consist of all the $H$-convex sets with property $P$ in $\mathbb{R}^d$, and let the sets in the system correspond to the $H$-convex sets in $\B$. An element $v$ of the base set is contained in a set $s$ of the system, if $v \subseteq s$ holds for the corresponding $H$-convex sets.

	\bibliographystyle{alpha}
	\bibliography{biblio}
	
\end{document}